\documentclass[11pt]{amsart}

\usepackage{amsfonts,amssymb}

\newtheorem{theorem}{Theorem}[section]
\newtheorem{proposition}[theorem]{Proposition}

\theoremstyle{remark}
\newtheorem{remark}[theorem]{Remark}
            
\newtheorem{example}[theorem]{Example}

\newcommand{\id}{\mathrm{id}}

\newcommand{\bee}[1]{\begin{equation}\label{#1}}
\newcommand{\beq}[1]{\begin{eqnarray}\label{#1}}
\newcommand{\ene}{\end{equation}}
\newcommand{\eqe}{\end{eqnarray}}

\begin{document}

\title[Lie Superautomorphisms]{Lie Superautomorphisms on Associative Algebras, II}

\author[Bahturin]{Yuri Bahturin}
\address{Department of Mathematics and Statistics\\Memorial
University of Newfoundland\\ St. John's, Canada}
\email{yuri@math.mun.ca}

\author[Bre\v{s}ar]{Matej Bre\v{s}ar}
\address{Faculty of Mathematics and Physics, 
University of Ljubljana,  and 
  Faculty of Natural Sciences and Mathematics,  University of Maribor,  Slovenia}
\email{matej.bresar@fmf.uni-lj.si}

\author[\v Spenko]{\v Spela \v Spenko}
\address{Faculty of Mathematics and Physics,
University of Ljubljana, Slovenia}
\email{spela.spenko@student.fmf.uni-lj.si}

\thanks{The first author was partially supported by
NSERC grant \# 227060-04 and URP grant, Memorial University of
Newfoundland. The second author was partially supported by ARRS grant  \#  P1-0288.}


\begin{abstract}
Lie superautomorphisms of prime associative superalgebras are considered.  A definitive result is obtained for central simple  superalgebras:  their Lie superautomorphisms are of  standard forms, except when the dimension of the superalgebra in question is $2$ or $4$. 
\end{abstract}

\maketitle

\section{Introduction}

Herstein's problems on Lie homomorphisms in associative algebras \cite{Her1} were solved in a long series of papers by making use of the  theory of functional identities; see \cite{FIbook}.
In view of the growing importance of superalgebras it is natural to pose  ``Herstein's problems for superalgebras", especially since various other problems from Herstein's program relating associative, Lie and Jordan structure in associative algebras have been recently considered, by several authors, in the super setting; see, e.g., \cite{GS, LS, Mont, Z} and references therein.

% For references and more details about the history and motivation we refer the reader to the preceding paper \cite{BB2}. 

The first question one can ask is whether the description of Lie isomorphisms of prime associative algebras \cite{Bre} can be extended to superalgebras. First attempts to solve this problem were made in the recent papers \cite{BB2} and \cite{W}.  In \cite{BB2} the problem was transfered  to the non-super setting through the Grassmann envelope. However, some information is lost when making this transfer and the results obtained in this way are not optimal; also,
 \cite{BB2} deals only with superalgebras in which  the $\mathbb Z_2$-grading is induced by an idempotent. The approach in \cite{W} is based on the study of certain  functional identities in superalgebras. This more direct method, similar to the one used in the  non-super case, has a disadvantage that it does not work in superalgebras whose odd elements are algebraic of a certain bounded degree.

 In this paper we propose a third approach which in particular yields a definitive answer for central simple  associative superalgebras. Actually, mostly we consider prime superalgebras, and come quite close to  the solution in this more general context. We consider separately three cases. The first one is where the grading is induced by an idempotent. Here our proof relies heavily on Martindale's arguments from his 1969 paper \cite{Mart} on Lie isomorphisms of rings with idempotents. The second case is where the grading is induced by an X-outer automorphism, and here the results on generalized functional identities with automorphisms \cite{GFI} are applicable.  The simplest case is the third one treating the situation where the grading is induced by the exchange automorphism. Besides applying the description of Jordan homomorphisms onto prime rings, the consideration in this case is self-contained and easy. Gathering together  information obtained in the treatment of these three cases we will be able to show that 
Lie superautomorphisms of   central simple  superalgebras are of  standard forms, except in algebras of dimension $2$ or $4$. Examples showing the necessity of  these two exclusions are given. Let us point out that our results hold  in both finite-dimensional and infinite-dimensional situation.

\section{Preliminaries}

In this section we introduce the notation, recall all necessary definitions, and mention some folklore results.
We begin by    fixing a field $F$ with the only restriction that 
$$ {\rm char} (F)\ne 2.$$
 All algebras in this paper will be algebras over $F$.

 Recall that a {\em  superalgebra} is a $\mathbb Z_2$-graded (nonassociative) algebra $A$; thus, $A= A_0\oplus A_1$ with $A_iA_j \subseteq A_{i+j}$
for all $i,j\in \mathbb Z_2$. Elements from $A_i$ are said to be {\em homogeneous of degree $i$}, $i=0,1$.  For $x \in A_i$ we  write $|x|=i$. We also call elements from $A_0$ {\em even} elements, and those from $A_1$ {\em odd} elements. A linear  subspace $V$ of $A$ is said to be {\em graded} if $V = V_0\oplus V_1$ where $V_i= A_i\cap V$. If $U$ and $V$ are graded spaces, then we say that  $\varphi:U\to V$ is a {\em graded map} if $\varphi(U_i)\subseteq V_i$, $i=0,1$.

A $\mathbb Z_2$-grading of $A$ can be equivalently expressed through an automorphism $\sigma$ of $A$, $x\mapsto x^\sigma$, such that $\sigma^2 =\id$. Indeed, if $A$ is $\mathbb Z_2$-graded, then one defines $\sigma$ by 
$$(x_0+x_1)^\sigma = x_0-x_1,\quad x_i\in A_i.$$
 Conversely, given an automorphism $\sigma$ of an algebra $A$ such that  $\sigma^2 =\id$, we see that $A$ becomes a superalgebra by defining
$A_0=\{x\in A\,|\, x^\sigma = x\}$ and $A_1=\{x\in A\,|\, x^\sigma = -x\}$. Every $x\in A$ can be written as 
{$x = x_0 + x_1$, where $x_0 = \frac{1}{2}(x + x^\sigma)\in A_0$ and  $x_1 = \frac{1}{2}(x- x^\sigma)\in A_1$.
 We shall say that $\sigma$ {\em induces the grading} on $A$.

An {\em associative superalgebra} is nothing but a superalgebra which is associative as an algebra. If $A$ is such an algebra, then we define the {\em supercommutator} of two  
 homogeneous elements $x,y\in A$ by
  $$[x,y]_s = xy - (-1)^{|x||y|}yx,$$
and then extend $[\,.\,,\,.\,\,]_s$ by bilinearity to $A\times A$ (we keep the notation $[\,.\,,\,.\,\,]$ for the usual commutator in algebras). One can check that for all  homogeneous elements $x,y,z\in A$ we have
 $$   [x,y]_s=-(-1)^{|x| |y|}[y,x]_s $$
 and
  $$
(-1)^{|x||z|}[[x,y]_s,z]_s + (-1)^{|z||y|}[[z,x]_s,y]_s + (-1)^{|y||x|}[[y,z]_s,x]_s = 0
$$
 That is, the super-anti\-commutativity   and the super-Jacobi identity hold for the product $[\,.\,,\,.\,\,]_s$. By the very definition this means that  $A$, 
    endowed with this new product and together with the original grading and the original vector space structure, becomes a {\em Lie superalgebra}. If $B$ is another associative superalgebra,
    then a graded linear map $\varphi:B\to A$ is said to be a {\em Lie superhomomorphism} if it satisfies
    $$\varphi([x,y]_s) = [\varphi(x),\varphi(y)]_s\quad\mbox{for all $x,y\in B$.}$$
    There are three standard examples of such maps. The first one is a {\em superhomomorphism}: this is a usual algebra homomorphism  that is also a graded map. The second one
is the negative of a superantihomomorphism; a {\em superantihomomorphism} between superalgebras is a graded linear map $\psi$ satisfying $\psi(xy) = (-1)^{|x||y|}\psi(y)\psi(x)$ for all homogeneous elements $x$ and $y$. The third example is of a different nature. 
Recall that the {\em supercenter} of $A$ is defined as the set of all $a\in A$ such that $[a,A]_s =0$.
Every graded linear map $\tau$ from $B$ into the supecenter of $A$ that vanishes on all supercommutators $[x,y]_s$, $x,y\in B$, is obviously a Lie superhomomorphism.
%(if $A_1$ has trivial intersection with the supercenter, then $\tau$ must, as a graded map,  vanish also on $B_1$).
 Moreover, if $\tau$ is such a map and $\theta:B\to A$ is either a superhomomorphism or the negative of a superantihomomorphism, then 
$$\varphi = \theta + \tau$$ 
is a Lie superhomomorphism. We say that such  Lie superhomomorphisms are of a {\em standard form}. 
Our basic goal in this paper is to show that, under appropriate assumptions, {\em Lie superautomorphisms} of associative superalgebras are of standard forms. In Sections   \ref{secfirst} and \ref{secsecond}, however, we will deal with {\em Lie superisomorphisms} between different superalgebras, primarily only because this is more  convenient   in those settings.

 Let $Z$ be the  usual center of an associative superalgebra $A$, i.e., the center of an algebra $A$. Note that $Z$ is a graded subspace of $A$, thus $Z=Z_0\oplus Z_1$. Clearly, $Z_0$ is contained in the supercenter of $A$, and in fact quite often $Z_0$ is equal to the supercenter (see below). 
 We say that $A$ is a {\em central superalgebra} if $Z_0=F$, i.e, $A$ 
is unital and $Z_0$ consists of scalar multiplies of $1$.

 An associative superalgebra $A$ is said to be {\em simple} if $A^2\ne 0$ and $A$ has no graded ideals different from $0$ and $A$. 
 More generally, $A$ is {\em prime} if the product of any two nonzero graded ideals of $A$ is  nonzero. 
 If $A$ is simple (resp. prime) as an algebra, then it is of course also simple (resp. prime) as a superalgebra.
 The converse is not true. If $U$ is a simple (resp. prime) associative algebra, then the direct product
  $A=U\times U$  of two copies of $U$ is not prime as algebra, but it is simple (resp. prime) 
  as a superalgebra defined through the exchange automorphism $(x,y)^\sigma = (y,x)$, i.e.,  $A_0= \{(u,u)\,|\, u\in U\}\cong U$ 
  and $A_1= \{(u,-u)\,|\, u\in U\}$. Anyway, we shall mostly  consider the situation where a prime superalgebra 
  $A$ is also prime as an algebra. Let us therefore recall some standard facts  about such algebras. 
  For all details we refer the reader to the book \cite{BMMb}.

  Until further notice we assume that $A$ is a prime associative algebra. Then one can form its {\em maximal left algebra of quotients} $Q$. This is a unital prime algebra containing $A$ as its subalgebra. The center $C$ of $Q$ is a field, called the {\em extended centroid of $A$}. Of course, $C$ contains the base field $F$, and moreover, it contains the  center $Z$ of $A$. By $A_C$ we denote the subalgebra of $Q$ generated by $A$ and $C$; thus $A_C$ consists of elements of the form $\lambda_0 + \sum \lambda_ix_i$ where $\lambda_i\in C$ and $x_i\in A$. If $A$ is a simple unital algebra, then $C =Z$, and hence $A_C=A$.
  
  In our proof in Section \ref{secfirst} we shall arrive at a situation where $a,b,c \in A$ with $a\ne 0$  satisfy 
  \begin{equation}\label{eJer}
  axb = cxa \quad\mbox{for all $x\in A$.}
  \end{equation}  
This is a typical situation where the extended centroid can be  effectively used. Namely, by a well-known result by Martindale \cite[Theorem 2]{Mart2} it follows that there exists $\lambda\in C$ such that  $b=c = \lambda a$. 

  Every automorphism  of $A$ can be uniquely extended to an automorphism of $Q$ \cite[Proposition 2.5.3]{BMMb}.  Moreover, if its square is the identity, then the same holds true for this extension. Accordingly, if $A$ is a superalgebra, then so is $Q$, and $A$ is its subsuperalgebra. Further, $C$ is a graded subspace of $Q$, $C= C_0\oplus C_1$.  Therefore $A$ is also a subsuperalgebra of $A_C$. We will also deal with the  subalgebra of $Q$ generated  by $A$ and the field $C_0$; we denote it by $A_{C_0}$. Clearly, $A_{C_0}$ is a superalgebra and $A$ is its subsuperalgebra.
   Let us remark that even when dealing with the usual (not super) Lie isomorphisms between prime algebras, one cannot avoid the extended centroid and related concepts in the structure theorems. The same happens in the super setting. In particular, the term ``standard form" introduced above should be understood  somewhat loosely, by allowing $\theta$ and $\tau$ to have their ranges in $A_{C_0}$ (i.e., not necessarily in $A$).
Moreover, it is easy to see that the supercenter of a prime associative superalgebra is equal to the even part of its center (see, e.g., \cite[Lemma 1.3]{Mont}).  Therefore $\tau$, as a graded map,
must vanish on the odd part.
 
An automorphism $\sigma$ of $A$ is said to be $X$-{\em inner} if there exists an invertible element $q\in Q$ such that $x^\sigma = qxq^{-1}$ for all $x\in A$. It turns out that $q$ must necessarily lie in the symmetric Martindale algebra of quotients $Q_s$ of $A$. This is a subalgebra of $Q$ that also contains $A$, and moreover, if $A$ is a simple unital algebra, then $Q_s =A$. Therefore the concept of an $X$-inner automorphism in this case coincides with the usual concept of an inner automorphism.  
If $\sigma$ is not $X$-inner, then it is called $X$-{\em outer}.

In Section \ref{secsecond} we will consider non-GPI prime algebras, i.e., algebras that do not satisfy  generalized polynomial identities. Again we refer to \cite{BMMb} for a full account on these notions.

\section{First case: grading incuded by an idempotent} \label{secfirst}

Let $A$ be an associative algebra and $e$ an idempotent in $A$.  If we set
\begin{equation}\label{e}
A_0 =eAe + (1-e)A(1-e)\quad\mbox{ and}\quad A_1 = eA(1-e) + (1-e)Ae,
\end{equation}
 then $A$ becomes an associative superalgebra. We shall say that in such a superalgebra the {\em grading is induced by an idempotent}. This type of grading is  important and it  appears quite often  (a glance at the proof of Theorem   \ref{Tsimple} below reveals the reasons for that).
 A prototype example  is $M(p\,|\,q)$, the algebra of square matrices of order $p+q$ equipped with the following  $\mathbb Z_2$-grading: $M(p\,|\,q)_0$ consists of matrices of the form 
$\left[\begin{matrix} A & 0 \cr 0 & D \cr
\end{matrix} \right]$, $A\in M_p(F)$, $D\in M_q(F)$, and $M(p\,|\,q)_1$ consists of matrices of the form 
$\left[\begin{matrix} 0 & B \cr C & 0 \cr
\end{matrix} \right]$, $B\in M_{p,q}(F)$, $C\in M_{q,p}(F)$.

 Our goal in this section is to describe Lie superisomorphisms between prime associative (possibly infinite-dimensional) algebras whose superalgebra structure is arising from idempotents. For simple algebras satisfying some technical conditions (partially redundant, as one can see from what follows), this was done in
\cite[Corollary 3.2]{BB2}.  We will now obtain a definitive result for prime algebras, using the approach developed by Martindale \cite{Mart}. For this particular problem this approach has turned out to be  more efficient than the one based on functional identities, used in \cite{BB2}.

Although Martindale considered Lie isomorphisms (not superisomorphisms), many of his arguments make sense in the present context. A careful inspection of the proofs of Lemmas 11, 12, 13, 14, 15, 16, 18, 19, 21, and Theorems 8, 9, 10 from \cite{Mart} shows that he actually proved the following proposition, which  is applicable to both Lie isomorphisms and Lie superisomorphisms. 
Its formulation is rather technical, but this is exactly what the proofs of the aforementioned lemmas and theorems, practically without any change, show.

\begin{proposition}\label{PJerry} {\bf (Martindale)}
Let $A$ and $B$ be prime associative unital algebras, and let $\varphi:B\to A$ be a bijective linear map.  Suppose there exist nontrivial idempotents  $e\in A$ and $f\in B$ such that 
$$\varphi(f) - e \in C \quad\mbox{and}\quad \varphi(1-f) - (1-e) \in C,$$
 where $C$ is the extended centroid of $A$. Let $B_0 = fBf+(1-f)B(1-f)$. If $\varphi$ satisfies 
$$\varphi([x_0,x]) = [\varphi(x_0),\varphi(x)]\quad\mbox{for all $x_0\in B_0$, $x\in B$,}$$
then there exist linear maps $\theta:B\to A_C$ and $\tau:B\to C$ such that $\varphi =\theta + \tau$,
$$\tau(fB(1-f)) = \tau((1-f)Bf) =0,$$
$$\theta(fB(1-f))\subseteq eA(1-e),\,\,\theta((1-f)Bf)\subseteq (1-e)Ae, $$
and either 
\begin{itemize}
\item[(a)] $\theta(fBf)\subseteq eA_C e$, $\theta((1-f)B(1-f))\subseteq (1-e)A_C (1-e)$, and  for all $x_0\in B_0$, $x\in B$ we have 
$\theta(x_0x) = \theta(x_0)\theta(x)$,  $\theta(xx_0) = \theta(x)\theta(x_0)$; or
\item[(b)] $\theta(fBf)\subseteq (1-e)A_C (1-e)$, $\theta((1-f)B(1-f))\subseteq eA_C e$, and  for all $x_0\in B_0$, $x\in B$ we have 
$\theta(x_0x) = -\theta(x)\theta(x_0)$,  $\theta(xx_0) = -\theta(x_0)\theta(x)$.
\end{itemize}
\end{proposition}

Let us point out that Martindale's main result on Lie isomorphisms also requires the assumption that the characteristic is different from $3$. But this is used only at one point, in the proof of \cite[Theorem 7]{Mart}. We will  derive a similar conclusion to the one from this theorem in a different manner, without any restriction on the characteristic.

\begin{theorem}\label{TJerry}
Let $A$ and $B$ be  prime associative unital algebras, and assume that $A$ and $B$ are also superalgebras with respect to  gradings  induced by idempotents. 
Let $C$ be the extended centroid of $A$ and assume that  $A_C$  is not isomorphic to $C$ or $M_2(C)$. Then 
every Lie superisomorphism $\varphi:B\to A$ is of the form $\varphi= \theta  + \tau$ where $\theta$ is either a superhomomomorphism  or the negative of a superantihomomorphism from $B$ into $A_C$, and $\tau$ is a linear map from $B$ into $C$  
satisfying $\tau([B,B]_s)  =0$.
\end{theorem}

\begin{proof}
Let $e\in A$ be an idempotent inducing the grading on $A$, and $f\in B$ be an idempotent inducing the grading on $B$. Clearly, $e$ is a trivial idempotent 
(i.e., $e=0$ or $e=1$) if and only if $A_1=0$, and $f$ is trivial if and only if $B_1=0$. Further, $A_1=0$ if and only if $B_1=0$,
 and in this case the result follows from \cite[Corollary 6.5]{FIbook}.   We may therefore assume that both $e$ and $f$ are nontrivial idempotents.

We set $u =\varphi(f)$. Since $[f,B_0] =0$ it follows that $[u,A_0] = [\varphi(f),\varphi(B_0)] =0$. In particular, $[u,eAe] =0$, i.e., $ exeu =uexe $ for all $x\in A$. This is an identity of the type \eqref{eJer}. Therefore 
 $eu = \alpha e$ for some $\alpha \in C$. Similarly,  $[u,(1-e)A(1-e)] =0$ yields $(1-e)u = \beta (1-e)$ for some $\beta\in C$. Accordingly, $u = \gamma e + \beta$ with $ \gamma =\alpha - \beta$.
Pick a nonzero $b\in fB(1-f)$, and set $c\ = \varphi(b)\in A_1$. We have $c = \varphi([f,b]) = [u,c] = \gamma [e,c]$. This  yields $ece=0$. Further, multiplying $c = \gamma [e,c]$ from the left (resp. right) by $e$ we get $(1-\gamma)ec =0$ (resp.   $(1+\gamma)ce =0$). Since $c\ne 0$, we have $[e,c]\ne 0$, and therefore $ec\ne 0$ or $ce \ne 0$. Accordingly, $\gamma = 1$ or $\gamma = -1$. We have thereby proved that
either $\varphi(f) - e \in C$ or $\varphi(f) - (1-e)\in C$. Now, saying that the grading on $A$ is induced by $e$ is equivalent to saying that it is induced by $(1-e)$. We may therefore replace the roles of $e$ and $1-e$ at the very beginning, and so there is no loss of generality in assuming that the first condition, $\varphi(f) - e \in C$, holds.
Next, since $1\in B_0$ and $[1,B] = 0$, it follows that $[\varphi(1),A]=0$, and hence $\varphi(1)\in C$. Accordingly, we have 
\begin{equation} 
\nonumber
\varphi(f) - e \in C\quad\mbox{and}\quad \varphi(1-f) - (1-e) \in C.
\end{equation}
We are now in a position to use Proposition \ref{PJerry}. Thus, $\varphi= \theta+\tau$ where $\tau:B\to C$ satisfies $\tau(B_1)=0$ and $\theta$ satisfies either (a) or (b).

For each $x\in B$ we write 
$$x_{11} = fxf,\,\, x_{12} = fx(1-f),\,\, x_{21} = (1-f)xf,\,\, x_{22} = (1-f)x(1-f). $$
Note that for all $x,y\in B$ we have
\begin{equation*}
\begin{split}
\theta(x_{12}y_{21} + y_{21}x_{12}) &= \varphi(x_{12}y_{21} + y_{21}x_{12}) - \tau(x_{12}y_{21} + y_{21}x_{12})\\
&= \varphi(x_{12})\varphi(y_{21}) + \varphi(y_{21})\varphi(x_{12}) - \tau(x_{12}y_{21} + y_{21}x_{12}).
\end{split}
\end{equation*}
Since $\tau(B_1)=0$, $\varphi$ and $\theta$ coincide on $B_1$. Setting 
$$
\varepsilon(x,y) =- \tau(x_{12}y_{21} + y_{21}x_{12})\in C
$$
we can 
therefore  rewrite the last identity as
\begin{equation} \label{elas}
\theta(x_{12}y_{21} + y_{21}x_{12}) = \theta(x_{12})\theta(y_{21}) + \theta(y_{21})\theta(x_{12}) +\varepsilon(x,y).
\end{equation}

Let us first consider the case where $\theta$ satisfies (a). 
Multiplying \eqref{elas} by $e$ and using the conclusions about $\theta$  from Proposition \ref{PJerry} we obtain
\begin{equation} \label{e12}
\theta(x_{12}y_{21}) = \theta(x_{12})\theta(y_{21}) + \varepsilon(x,y)e
\end{equation}
for all $x,y\in B$. Similarly we have
\begin{equation} \label{e21}
\theta(y_{21}x_{12}) = \theta(y_{21})\theta(x_{12}) + \varepsilon(x,y)(1-e)
\end{equation}
for all $x,y\in B$. 

Since $\theta$ satisfies (a), it follows from \eqref{e12} that, on the one hand, 
$$
\theta(x_{12}y_{21}z_{11}) = \theta(x_{12}y_{21})\theta(z_{11}) = \theta(x_{12})\theta(y_{21})\theta(z_{11}) + \varepsilon(x,y)e\theta(z_{11})
$$ 
and on the other hand,
$$
\theta(x_{12}y_{21}z_{11}) = \theta(x_{12})\theta(y_{21}z_{11}) + \varepsilon(x,yez)e = \theta(x_{12})\theta(y_{21})\theta(z_{11}) + \varepsilon(x,yez)e.
$$ 
Comparing we get $\varepsilon(x,y)e\theta(z_{11}) =  \varepsilon(x,yez)e$.
Accordingly, if  $\varepsilon(x,y)\ne 0$ for some $x,y\in B$, then $e\theta(z_{11}) \in Ce$ for every $z\in B$. Consequently,
$$
e\theta(z)e = e\bigl(\theta(z_{11}) + \theta(z_{12}) + \theta(z_{21}) + \theta(z_{22})\bigr)e = e\theta(z_{11})e= e\theta(z_{11}) \in Ce.
$$
Since $\varphi = \theta + \tau$ and $\varphi$ is surjective, this implies $eAe\subseteq Ce$, and hence $eA_C e = Ce$. Similarly,
by making use of \eqref{e21}, we see that $\varepsilon(x,y)\ne 0$ yields $(1-e)A_C (1-e) = C(1-e)$. However, we claim that the conditions $eA_C e = Ce$ and 
$(1-e)A_C (1-e) = C(1-e)$ together imply $A_C\cong M_2(C)$. Indeed, set $u_{11} = e$, $u_{22} = 1-e$. Since $0\ne eA_C(1-e)A_Ce\subseteq Ce$ we can find $a,a'\in A_C$ such that 
$ea(1-e)a'e = e$. Setting $u_{12} = ea(1-e)$ and $u_{21} = (1-e)a'e$ we thus have $u_{12}u_{21} = u_{11}$. As $u_{21}u_{12}\in (1-e)A(1-e)$, there exists $\alpha\in C$ such that $u_{21}u_{12} = \alpha(1-e)$. Multiplying from the left by $e_{12}$ it follows that $ u_{12} =\alpha u_{12}$, and so $\alpha = 1$; that is,  $u_{21}u_{12} = u_{22}$. We have thereby showed that $A_C$ contains a set of $2\times 2$ matrix units $\{u_{ij}\,|\, i,j=1,2\}$ such that $u_{11}A_C u_{11}\cong C$. As it is well known and easy to see, this implies our claim $A_C\cong M_2(C)$. Since $A$ does not satisfy this condition by our assumption, we must have $\varepsilon(x,y)= 0$ for all $x,y\in B$. Thus, \eqref{e12} and \eqref{e21}
reduce to $\theta(x_{12}y_{21}) = \theta(x_{12})\theta(y_{21})$ and $\theta(y_{21}x_{12}) = \theta(y_{21})\theta(x_{12})$. Since $\theta(x_{ij}y_{ij}) = 0 =\theta(x_{ij})\theta(y_{ij})$ trivially holds for $i\ne j$, gathering together all information about $\theta$ we see that it is a superhomomorphism.  

In a similar fashion one shows that $\theta$ is the negative of a superantihomomorphism in the case where (b) holds. 

It is now immediate to check that $\tau = \varphi - \theta$ satisfies $\tau([B,B]_s)  =0$.
\end{proof}

\begin{remark} \label{rem1}
It is easy to see that $\theta$ is injective. But we cannot say much about its range, not even in the case of trivial idempotents. 
See, e.g., \cite[Example 6.10]{FIbook}.
\end{remark}

\begin{remark} \label{rem111}
In the setting studied in Theorem \ref{TJerry}, $C_0$ coincides with $C$, and hence $A_{C} = A_{C_0}$. 
%This is the reason why the stamenent of the result in the next section is apparently different.
\end{remark}

In results on (usual) Lie isomorphisms of prime algebras there is no need to exclude some special types of algebras. The next example justifies the exclusion of the $2\times2$ matrix algebra in Theorem \ref{TJerry}.

\begin{example} \label{Exe}
Let $A=M(1\,|\,1)$. Then the extended centroid $C$ of $A$ coincides with $F$, and $A=A_C$.
%Consider $A$ as a superalgebra with respect to the grading induced by the matrix unit $e=e_{11}$.
 Note that $\varphi:A\to A$ defined by
$$
\varphi\left(\left[\begin{matrix} x & y \cr z & w \cr \end{matrix}\right]\right)=  \left[\begin{matrix} 2x & 2y \cr z & x+w \cr \end{matrix}\right]
$$
is a Lie superautomorphism. Suppose  $\varphi$  was of the form  $\varphi = \theta + \tau$ where $\theta$ and $\tau$ are as in Theorem \ref{TJerry}.  As $\tau$ is a central map vanishing on $[A,A]_s$,
we have
 $$
\tau\left(\left[\begin{matrix} x & y \cr z & w \cr \end{matrix}\right]\right)=  \left[\begin{matrix} c(x-w) & 0 \cr 0 & c(x-w) \cr \end{matrix}\right]
$$
for some $c\in F$, and hence
$$
\theta\left(\left[\begin{matrix} x & y \cr z & w \cr \end{matrix}\right]\right)=  \left[\begin{matrix} (2- c)x+ cw & 2y \cr z & (1-c)x + (1+c)w  \cr \end{matrix}\right].
$$
However, $\theta(e_{12}e_{21})$ is equal neither to $\theta(e_{12})\theta(e_{21})$ nor to  $\theta(e_{21})\theta(e_{12})$, a contradiction.
\end{example}

\section{Second case: grading induced by an $X$-outer automorphism $\sigma$} \label{secsecond}

The result in this section will be obtained as an  application of the theory of functional identities.  We begin by introducing the necessary notation needed for dealing with functional identities.

Let $A$ be an algebra, and let  $x_1,\ldots,x_d\in A$. For $1\le i\le d$ we write 
$$\overline{x}_d^i =  (x_1,\ldots,x_{i-1},x_{i+1},\ldots,x_d) \in A^{d-1}=A\times\ldots\times A,$$
 and for $1 \le i < j \le d$ we write 
$$\overline{x}_d^{ij} =\overline{x}_d^{ji} =  (x_1,\ldots,x_{i-1},x_{i+1},\ldots,x_{j-1},x_{j+1},\ldots,x_d)\in A^{d-2}.$$
We will consider functions defined on $A^{d-1}$ and $A^{d-2}$. We identify a function defined on $A^{0}$ by a fixed  element from the range of this function. 

The following result, which is a very special case of \cite[Theorem 1.2]{GFI}, will be used in our proof. 

\begin{theorem} \cite{GFI} \label{Tgfi}
Let $A$ be a non-GPI prime algebra, let $Q$ be its maximal left algebra of quotients, let $C$ be the extended centroid of $A$, and let $V$ be a finite-dimensional subspace of the vector space $Q$ over $C$. Further, let $\sigma$ be be an $X$-outer automorphism of $A$, let $d\ge 2$, and let $E_i,G_i,F_j,H_j:A^{d-1}\to Q$, $1\le i,j\le d$, be functions such that
\begin{align*}
\begin{split}
\sum_{i=1}^d E_i(\overline{x}_d^i)x_i +  \sum_{i=1}^d  G_i(\overline{x}_d^i)x_i^\sigma + \sum_{j=1}^d x_jF_j(\overline{x}_d^j) +  \sum_{j=1}^d  x_j^\sigma H_j(\overline{x}_d^j)\in V
\end{split}
\end{align*}
for all $x_1,\ldots,x_d\in A$. Then there exist unique functions 
\begin{eqnarray*}
&\mbox{$p_{ij},q_{ij},r_{ij},s_{ij}:A^{d-2}\to Q$, $1\le i,j\le d$, $i\ne j$, and}\\
&\mbox{$\lambda_i,\mu_i: A^{d-1}\to C$, $1\le i\le d$,}
\end{eqnarray*}
 such that
\begin{align*}
\begin{split}
E_i(\overline{x}_d^i) & = \sum_{j=1\atop j\not=i}^d x_j p_{ij}(\overline{x}_d^{ij})  + \sum_{j=1\atop j\not=i}^d x_j^\sigma r_{ij}(\overline{x}_d^{ij})  + \lambda_i(\overline{x}_d^i),\\
G_i(\overline{x}_d^i) & = \sum_{j=1\atop j\not=i}^d x_j q_{ij}(\overline{x}_d^{ij})  + \sum_{j=1\atop j\not=i}^d x_j^\sigma s_{ij}(\overline{x}_d^{ij})  + \mu_i(\overline{x}_d^i),\\
F_j(\overline{x}_d^j) & = -\sum_{i=1\atop i\not=j}^d  p_{ij}(\overline{x}_d^{ij})x_i  - \sum_{i=1\atop i\not=j}^d  q_{ij}(\overline{x}_d^{ij})x_i^\sigma  - \lambda_j(\overline{x}_d^j),\\
H_j(\overline{x}_d^j) & = -\sum_{i=1\atop i\not=j}^d  r_{ij}(\overline{x}_d^{ij})x_i  - \sum_{i=1\atop i\not=j}^d  s_{ij}(\overline{x}_d^{ij})x_i^\sigma  - \mu_j(\overline{x}_d^j)
\end{split}
\end{align*}
for all $x_1,\ldots,x_d\in A$. 
 In particular,
\begin{align*}
\begin{split}
\sum_{i=1}^d E_i(\overline{x}_d^i)x_i +  \sum_{i=1}^d  G_i(\overline{x}_d^i)x_i^\sigma + \sum_{j=1}^d x_jF_j(\overline{x}_d^j) +  \sum_{j=1}^d  x_j^\sigma H_j(\overline{x}_d^j) =0.
\end{split}
\end{align*} 
Moreover, if all $E_i,G_i,F_j,H_j$ are multilinear, then so are $p_{ij},q_{ij},r_{ij},s_{ij},\lambda_i,\mu_i$.
\end{theorem}

\begin{remark}\label{remGFI}
The uniqueness of the $p_{ij}$'s, $q_{ij}$'s etc. implies the following: If all $F_j,H_j$ are $0$, then all $E_i,G_i$ are $0$. Similarly, if all $E_i,G_i$ are $0$, then  all $F_j,H_j$ are $0$. See also \cite[Theorems 3.1 and 3.2]{GFI}.
\end{remark}

\begin{remark}\label{remGFI2}
Under assumptions of Theorem \ref{Tgfi}, assume that $a,b\in A$ are such that either $a(x+x^\sigma)b = 0$ for all $x\in A$, or  $a(x-x^\sigma)b = 0$ for all $x\in A$. Then $a=0$ or $b=0$. This follows immediately from  
a slightly different version of \cite[Theorem 1.2]{GFI} than the one
stated above. On the other hand,  these are
 very simple examples of  generalized polynomial identities with an $X$-outer automorphism, for which Kharchenko's theory easily gives this conclusion   (cf. \cite[Chapter 7]{BMMb}). This theory is also used in the proof of \cite[Theorem 1.2]{GFI}.
\end{remark}

We will use  Theorem \ref{Tgfi} only for  $d=2$ and $d=3$. The case where  $V =0$ is the one that is most commonly used, but we shall also arrive at other subspaces in the course of the proof of the next theorem.

\begin{theorem} \label{TXout}
Let $A$ and $B$ be  associative superalgebras. Assume  that, as an algebra, $A$ is a non-GPI prime algebra, and assume  the grading of   $A$ is  induced by the $X$-outer automorphism $\sigma$. 
Then every Lie superisomorphism $\varphi:B\to A$ is of the form $\varphi= \theta  + \tau$ where $\theta$ is either a superhomomomorphism  or the negative of a superantihomomorphism from $B$ into $A_{C_0}$, and $\tau$ is a linear map from $B$ into $C_0$  
satisfying $\tau([B,B]_s)  =0$.
\end{theorem}

\begin{proof}
For any $x,y\in A$ we set
$$F(x,y) =
\varphi(\varphi^{-1}(x)\varphi^{-1}(y)).$$
Note that $F$ satisfies 
\begin{equation}\label{Spela1}
F(F(x,y),z) = F(x,F(y,z)) \quad\mbox{for all $x,y,z\in A.$}
\end{equation}
Further, since $\varphi$ is a Lie superisomorphism, we have $F(A_i,A_j)\subseteq A_{i+j}$ for all $i,j\in\mathbb Z_2$,
\begin{equation}\label{Spela2}
F(x_0,y) = F(y,x_0) + [x_0,y]  \quad\mbox{for all $x\in A_0$, $y\in A,$}
\end{equation}
and
\begin{equation}\label{Spela3}
F(x_1,y_1) = - F(y_1,x_1) + x_1y_1 + y_1x_1  \quad\mbox{for all $x_1,y_1\in A_1$.}
\end{equation}

We will now derive a functional identity involving $F$, for which Theorem \ref{Tgfi} is applicable. We begin by noticing that
 $$
(-1)^{|u||w|}[uv,w]_s + (-1)^{|w||v|}[wu,v]_s + (-1)^{|v||u|}[vw,u]_s = 0
$$
holds for all homogeneous $u,v,w\in B$. Consequently,
$$
(-1)^{|u||w|}[\varphi(uv),\varphi(w)]_s + (-1)^{|w||v|}[\varphi(wu),\varphi(v)]_s + (-1)^{|v||u|}[\varphi(vw),\varphi(u)]_s = 0.
$$
  This readily yields
\begin{equation}\label{esu}
(-1)^{|x||z|}[F(x,y),z]_s + (-1)^{|z||y|}[F(z,x),y]_s +
(-1)^{|y||x|}[F(y,z),x]_s = 0
\end{equation}
for all homogeneous $x,y,z\in A$. Let us consider two particular cases of \eqref{esu}. Firstly, if $x=x_0\in A_0$ and $y=y_0\in A_0$, then \eqref{esu} becomes
$$F(x_0,y_0)z + F(z,x_0)y_0 + F(y_0,z)x_0  =
 zF(x_0,y_0) + y_0F(z,x_0) + 
x_0F(y_0,z)$$
for all $z\in A$.
Secondly, if $x=x_1\in A_1$ and $y=y_0\in A_0$, then \eqref{esu} becomes
$$
F(x_1,y_0)z^{\sigma} + F(z,x_1)y_0 + F(y_0,z)x_1 
=zF(x_1,y_0) + y_0F(z,x_1) + 
x_1F(y_0,z^{\sigma})
$$
for all $z\in A$. Adding together these two identities we obtain  
\begin{align*}
&F(x_0,y_0)z  + F(x_1,y_0)z^{\sigma} + F(z,x)y_0 + F(y_0,z)x \\
=& zF(x,y_0) + y_0F(z,x) + x_0F(y_0,z) + x_1F(y_0,z^{\sigma})
\end{align*}
for all $z\in A$, $x_0,y_0\in A_0$, $x_1\in A_1$, where $x=x_0 + x_1$. Let us replace $y_0$ by $y+y^\sigma$ in this identity, and similarly,
$x_0$ by $\frac{1}{2}(x+x^\sigma)$ and $x_1$ by $\frac{1}{2}(x-x^\sigma)$. Then we get
\begin{align*}
&\frac{1}{2}F(x+x^\sigma,y+y^\sigma)z  + \frac{1}{2}F(x-x^\sigma,y+y^\sigma)z^{\sigma} + F(z,x)y\\
 +&  F(z,x)y^\sigma + F(y+y^\sigma,z)x - zF(x,y+y^\sigma) -yF(z,x) - y^\sigma F(z,x)\\
  -& \frac{1}{2}xF(y+y^\sigma,z+z^\sigma) -\frac{1}{2}x^\sigma F(y+y^\sigma,z - z^{\sigma})  = 0
\end{align*}
for all $x,y,z\in A$. This is a type of a functional identity that is treated in Theorem \ref{Tgfi}. We shall not  need  the full force of this theorem.  Let us concentrate
only on terms  $F(z,x)y$ and $-yF(z,x)$ appearing in the identity. Theorem  \ref{Tgfi} tells us that, on the one hand, we have
$$
F(z,x) = xp_{1}(z) + zp_{2}(x) + x^\sigma r_{1}(z) + z^\sigma r_{2}(x) + \lambda(z,x),
$$
and, on the other hand, we have
$$
F(z,x) = p_{1}'(z)x + p_{2}'(x)z +  q_{1}(z)x^\sigma +  q_{2}(x)z^\sigma + \lambda(z,x),
$$
where $p_{i},p_i',r_{i},q_{i}:A\to Q$ are linear maps and $\lambda:A^2 \to C$ is a bilinear map. Comparing both expressions we get
$$
xp_{1}(z) + zp_{2}(x) + x^\sigma r_{1}(z) + z^\sigma r_{2}(x) =  p_{1}'(z)x + p_{2}'(x)z +  q_{1}(z)x^\sigma +  q_{2}(x)z^\sigma.
$$
We may now use Theorem \ref{Tgfi} once again, this time for $d=2$. Hence we see that, in particular,  $p_{1}$ can be expressed as  $p_{1}(z) = a_1z+ a_1'z^\sigma + \gamma_1(z)$ for some $a_1,a_1'\in Q$
and a linear map $\gamma_1:A\to C$. Similarly we can express other functions. Hence  we can conclude that there  exist 
$a_{ij},b_{ij}\in Q$,  linear maps $\lambda_i,\mu_i:A\to C$ and a bilinear map $\lambda:A^2\to C$ such that
\begin{align} 
F(z,x) &= za_{11}x + za_{12}x^\sigma + z^\sigma a_{21}x + z^\sigma a_{22}x^\sigma  \nonumber \\
&+ xb_{11}z + xb_{12}z^\sigma + x^\sigma b_{21}z + x^\sigma b_{22}z^\sigma  \label{spela4}\\
&+ \lambda_1(x)z+\lambda_2(x)z^\sigma + \mu_1(z)x + \mu_2(z)x^\sigma +  \lambda(z,x) \nonumber
\end{align}
for all $z,x\in A$.
Setting $a= a_{11} + a_{12}+ a_{21} + a_{22}$, $b= b_{11} + b_{12}+ b_{21} + b_{22}$, $\omega = \lambda_1 + \lambda_2$ and $\omega' = \mu_1+\mu_2$ we get
\begin{equation*}%\label{dod}
F(z_0,x_0) = z_0ax_0 + x_0bz_0 + \omega(x_0)z_0 + \omega'(z_0)x_0 + \lambda(z_0,x_0)
\end{equation*}
for all $z_0,x_0\in A_0$. Since $F(z_0,x_0) - F(x_0,z_0) = z_0x_0 - x_0z_0$ by \eqref{Spela2}, we thus have
$$
z_0(a-b-1)x_0 - x_0(a-b-1)z_0 + (\omega - \omega')(x_0)z_0 - (\omega - \omega')(z_0)x_0\in C. 
$$
Writing $z+z^\sigma$ for $z_0$ and  $x+x^\sigma$ for $x_0$   we obtain
\begin{align*}
&\bigl(z_0(a-b-1)  - (\omega - \omega')(z_0)\bigr)x + \bigl( z_0(a-b-1)- (\omega - \omega')(z_0)\bigr)x^\sigma \\ 
- & \bigl(x_0(a-b-1) - (\omega - \omega')(x_0)\bigr)z - \bigl(x_0(a-b-1) 
 - (\omega - \omega')(x_0)\bigr)z^\sigma 
\in C.
\end{align*}
Note that we have  arrived at a situation considered in Remark \ref{remGFI}. Hence it follows that $z_0(a-b-1) = (\omega - \omega')(z_0)\in C$, i.e.,
$$z(a-b-1) + z^\sigma(a-b-1)\in C.$$
 Using
Remark \ref{remGFI} once again we get $a = b+1$, and hence $\omega = \omega'$ on $A_0$.

Next we have
\begin{align*}
&F(F(z_0,y_0),x_0)\\
 =& \bigl(z_0ay_0 + y_0bz_0 + \omega(y_0)z_0 + \omega(z_0)y_0 + \lambda(z_0,y_0)\bigr)ax_0 \\
 +& x_0b\bigl(z_0ay_0 + y_0bz_0 + \omega(y_0)z_0 + \omega(z_0)y_0 + \lambda(z_0,y_0)\bigr)\\
  +& \omega(x_0)\bigl(z_0ay_0 + y_0bz_0 + \omega(y_0)z_0 + \omega(z_0)y_0 + \lambda(z_0,y_0)\bigr)\\
   +& \omega(F(z_0,y_0))x_0 + \lambda(F(z_0,y_0),x_0),
\end{align*}
and 
\begin{align*}
&F(z_0, F(y_0,x_0))\\
 =& z_0a\bigl(y_0ax_0 + x_0by_0 + \omega(x_0)y_0 + \omega(y_0)x_0 + \lambda(y_0,x_0)\bigr)\\
 + &\bigl(y_0ax_0 + x_0by_0 + \omega(x_0)y_0 + \omega(y_0)x_0 + \lambda(y_0,x_0)\bigr)bz_0\\
  +& \omega(z_0)\bigl(y_0ax_0 + x_0by_0 + \omega(x_0)y_0 + \omega(y_0)x_0 + \lambda(y_0,x_0)\bigr)\\
   +& \omega(F(y_0,x_0))z_0 + \lambda(z_0,F(y_0,x_0)).
\end{align*}
Since $F(F(z_0,y_0),x_0)=F(z_0, F(y_0,x_0))$ by \eqref{Spela1}, comparing both expressions we obtain
\begin{align}\label{nn}
&\Bigl(y_0bz_0a + \lambda(z_0,y_0)a + \omega(F(z_0,y_0)) - \omega(z_0)\omega(y_0) \Bigr) x_0\nonumber \\
+& \bigl(x_0bz_0a - z_0ax_0b\bigr)y_0  \\
-& \Bigl(y_0ax_0b +\lambda(y_0,x_0)b + \omega(F(y_0,x_0)) - \omega(x_0)\omega(y_0) \Bigr) z_0 \nonumber\\
+&  x_0( \lambda(z_0,y_0)b) - z_0(\lambda(y_0,x_0)a)  
 \in C.\nonumber
\end{align}
Substituting $x+x^\sigma$ for $x_0$ etc. we arrive at a functional identity of the type treated in Theorem \ref{Tgfi}. Hence it follows, in particular, that there are
functions $p_i,q_i:A\to Q$ and $\nu:A^2\to C$ such that
$$
\lambda(y+y^\sigma,x+x^\sigma)a = p_1(y)x + p_2(x)y + q_1(x)y^\sigma + q_2(y)x^\sigma  +\nu(x,y).
$$
Thus, $p_1(y)x + p_2(x)y + q_1(x)y^\sigma + q_2(y)x^\sigma$ always lies in the space $V=C+Ca$. From Remark \ref{remGFI} we infer that $p_i$ and $q_i$ are $0$. Consequently,
$\lambda(y+y^\sigma,x+x^\sigma)a = \nu(x,y)\in C$. In the same way we derive from \eqref{nn} that $ \lambda(z_0,y_0)b$ always lies in $C$. We can therefore rewrite \eqref{nn} as 
\begin{align}\label{lll}
&\Bigl(y_0bz_0a + \lambda(z_0,y_0)(a+b) + \omega(F(z_0,y_0)) - \omega(z_0)\omega(y_0) \Bigr) x_0\nonumber\\
+& \bigl(x_0bz_0a - z_0ax_0b\bigr)y_0  \\
-& \Bigl(y_0ax_0b +\lambda(y_0,x_0)(a+b) + \omega(F(y_0,x_0)) - \omega(x_0)\omega(y_0) \Bigr) z_0 \in C.\nonumber
\end{align}
Making the usual substitution $x+x^\sigma$ for $x_0$ etc. we see that Remark \ref{remGFI} can be used. In particular it follows that
$$
(x+x^\sigma)b(z+z^\sigma)a - (z+z^\sigma)a(x+x^\sigma)b =0.
$$
Using Remark \ref{remGFI} again we obtain $b(z+z^\sigma)a =0$, yielding  $a=0$ or $b=0$ by  Remark \ref{remGFI2}.
Since
$a = b+1$ we actually have $a=0$ and $b=-1$, or $a=1$ and $b=0$. We will consider only the first possibility. As we shall see, it will lead to the conclusion that $\varphi$ can be expressed through the negative of a superantihomomorphism. The second possibility where $a=1$ and $b=0$ corresponds to the superhomomorphism case.

Thus, assume that $a=0$ and $b=-1$. Therefore we have
\begin{equation}\label{dod2}
F(y_0,x_0) = - x_0y_0 + \omega(x_0)y_0 + \omega(y_0)x_0 + \lambda(y_0,x_0).
\end{equation}
Next, \eqref{lll} reduces to 
\begin{align*}\label{lll}
&\Bigl(-\lambda(z_0,y_0) + \omega(F(z_0,y_0)) - \omega(z_0)\omega(y_0) \Bigr) x_0\nonumber\\
-& \Bigl(-\lambda(y_0,x_0) + \omega(F(y_0,x_0)) -\omega(x_0)\omega(y_0) \Bigr) z_0 \in C.\nonumber
\end{align*}
A standard application of Remark \ref{remGFI} yields
\begin{equation}\label{ccc}
-\lambda(y_0,x_0) + \omega(F(y_0,x_0)) - \omega(x_0)\omega(y_0)=0.
\end{equation}

By \eqref{Spela2} we have $F(x_1,y_0) - F(y_0,x_1) - x_1y_0 + y_0x_1 =0$ for all $x_1\in A_1$, $y_0\in A_0$. In view of 
\eqref{spela4} we can rewrite this as follows:
\begin{align*}
&  x_1(a_{11}+a_{12}-a_{21} - a_{22} - b_{11} - b_{12} + b_{21} + b_{22} -1)y_0\\ 
+ & y_0(b_{11}-b_{12}+b_{21} - b_{22} - a_{11}+a_{12}-a_{21} + a_{22}+1)x_1\\
+ &(\lambda_1 - \lambda_2 -\mu_1 + \mu_2)(y_0)x_1 + (\mu_1 + \mu_2 -\lambda_1 -\lambda_2)(x_1)y_0 \\
= & \lambda(y_0,x_1)- \lambda(x_1,y_0).
\end{align*} 
We may now apply Remark \ref{remGFI} iteratively, first for $d=3$ and then for $d=2$,  following the already familiar procedure. In particular we then get $(\mu_1 + \mu_2 -\lambda_1 -\lambda_2)(A_1) =0$, showing that $\omega$ and $\omega'$ coincide
 on $A_1$ as well. We also obtain 
\begin{eqnarray}\label{ab1}
%&a_{11}+a_{12}-a_{21} - a_{22} - b_{11} - b_{12} + b_{21} + b_{22} -1 = 0,\nonumber\\
%&b_{11}-b_{12}+b_{21} - b_{22} - a_{11}+a_{12}-a_{21} + a_{22}+1 =0,\nonumber\\
&(\lambda_1 - \lambda_2 -\mu_1 + \mu_2)(A_0)=0.
%&(\mu_1 + \mu_2 -\lambda_1 -\lambda_2)(A_1) =0,\\
%&\lambda(y_0,x_1)= \lambda(x_1,y_0).\label{esym}
\end{eqnarray}

Similarly, applying \eqref{spela4} to  \eqref{Spela3}  we obtain
\begin{align*}
&  x_1(a_{11}-a_{12} -a_{21} + a_{22} + b_{11} - b_{12} - b_{21} + b_{22} -1)y_1\\ 
+ & y_1(b_{11}-b_{12}-b_{21} + b_{22} + a_{11} - a_{12}-a_{21} + a_{22}-1)x_1\\
+ &(\lambda_1 - \lambda_2 +\mu_1 - \mu_2)(y_1)x_1 + (\mu_1 - \mu_2 +\lambda_1 -\lambda_2)(x_1)y_1 \\
= & -\lambda(y_1,x_1)- \lambda(x_1,y_1),
\end{align*} 
which implies
\begin{eqnarray}
 &a_{11}-a_{12} -a_{21} + a_{22} + b_{11} - b_{12} - b_{21} + b_{22} -1 =0,\label{abab}\\
&(\lambda_1 - \lambda_2 +\mu_1 - \mu_2)(A_1)=0.\label{ab2}
%\\&\lambda(y_1,x_1)= - \lambda(x_1,y_1).\label{e1111}
\end{eqnarray}

Let us set $\rho = \lambda_1 - \lambda_2$. From \eqref{ab1} and \eqref{ab2} we see that $\rho$ coincides with $\mu_1 - \mu_2$ on $A_0$, and with  $\mu_2 - \mu_1$ on $A_1$. 
Further, let $c = a_{11} + a_{12} - a_{21} - a_{22}$, and $d =  b_{11} - b_{12} + b_{21} - b_{22}$. By \eqref{spela4} we have
\begin{equation}\label{cd1}
F(x_1,y_0) = x_1cy_0 + y_0dx_1 + \rho(y_0)x_1 + \omega(x_1)y_0 + \lambda(x_1,y_0)
\end{equation}
for all $x_1\in A_1$, $y_0\in A_0$. Since $F(y_0,x_1) = F(x_1,y_0) + [y_0,x_1]$, it follows that
\begin{equation}\label{cd2}
F(y_0,x_1) = x_1(c-1)y_0 + y_0(d+1)x_1 + \rho(y_0)x_1 + \omega(x_1)y_0 + \lambda(x_1,y_0).
\end{equation}
Further, setting $e=a_{11}-a_{12} - a_{21} + a_{22}$, and noticing that  $b_{11} - b_{12} - b_{21} + b_{22} = 1 - e$ by \eqref{abab}, we see from \eqref{spela4} that
\begin{equation}\label{cd3}
F(u_1,z_1) = u_1ez_1 + z_1(1-e)u_1 + \rho(z_1)u_1 - \rho(u_1)z_1 + \lambda(u_1,z_1)
\end{equation}
for all $u_1,z_1\in A_1$.

Let $y_0,z_0\in A_0$ and $x_1\in A_1$. Applying \eqref{cd1} we obtain
\begin{align*}
&F(F(x_1,y_0),z_0)\\
 =& \bigl(x_1cy_0 + y_0dx_1 + \rho(y_0)x_1 + \omega(x_1)y_0 + \lambda(x_1,y_0)  \bigr)c z_0 \\
 +& z_0d\bigl( x_1cy_0 + y_0dx_1 + \rho(y_0)x_1 + \omega(x_1)y_0 + \lambda(x_1,y_0)  \bigr)\\
  +& \rho(z_0)\bigl(x_1cy_0 + y_0dx_1 + \rho(y_0)x_1 + \omega(x_1)y_0 + \lambda(x_1,y_0)  \bigr)\\
   +& \omega(F(x_1,y_0))z_0 + \lambda(F(x_1,y_0),z_0).
\end{align*}
Similarly, using \eqref{dod2}  and \eqref{cd1} we get
\begin{align*}
&F(x_1,F(y_0,z_0))\\
 =& x_1c\bigl(- z_0y_0 + \omega(z_0)y_0 + \omega(y_0)z_0 + \lambda(y_0,z_0)\bigr)\\
  +& \bigl( - z_0y_0 + \omega(z_0)y_0 + \omega(y_0)z_0 + \lambda(y_0,z_0) \bigr)dx_1 \\
  +& \omega(x_1)\bigl(- z_0y_0 + \omega(z_0)y_0 + \omega(y_0)z_0 + \lambda(y_0,z_0)\bigr)\\
   +& \rho(F(y_0,z_0))x_1 + \lambda(x_1,F(y_0,z_0)).
\end{align*}
In view of \eqref{Spela1} we can equate these two expressions. We can now argue similarly as above, when equating $F(F(z_0,y_0),x_0)$ and $F(z_0, F(y_0,x_0))$. The necessary modifications in the argument are quite obvious, and so we just outline the procedure. First one notices that $\lambda(A_1,A_0)d\subseteq C$. Using Remark \ref{remGFI} then one  shows that
$c=0$ and $\omega(A_1)=0$, and that either $d=0$ or $d=-1$. As we shall see, the first possibility cannot occur. To show this, we
let $x_1,z_1\in A_1$ and $y_0\in A_0$,  and use \eqref{cd1} (with $c=0$ and $\omega(A_1) =0$)  and \eqref{cd3}  to obtain
 \begin{align*}
&F(F(x_1,y_0),z_1)\\
 =& \bigl( y_0dx_1 + \rho(y_0)x_1 +  \lambda(x_1,y_0) \bigr)ez_1 \\
 +& z_1(1-e)\bigl(y_0dx_1 + \rho(y_0)x_1 +  \lambda(x_1,y_0) \bigr)\\
  +& \rho(z_1)\bigl(y_0dx_1 + \rho(y_0)x_1 + \lambda(x_1,y_0) \bigr)\\
   -& \rho(F(x_1,y_0))z_1 + \lambda(F(x_1,y_0),z_1).
\end{align*}
Similarly, from  \eqref{cd2}  and \eqref{cd3} we get
 \begin{align*}
&F(x_1,F(y_0,z_1))\\
 =& x_1e\bigl(-z_1y_0 + y_0(d+1)z_1 + \rho(y_0)z_1 + \lambda(z_1,y_0)\bigr)\\
  +& \bigl(-z_1y_0 + y_0(d+1)z_1 + \rho(y_0)z_1 + \lambda(z_1,y_0)\bigr)(1-e)x_1 \\
  +& \rho(x_1)\bigl(-z_1y_0 + y_0(d+1)z_1 + \rho(y_0)z_1 +  \lambda(z_1,y_0)\bigr)\\
   -& \rho(F(y_0,z_1))x_1 + \lambda(x_1,F(y_0,z_1)).
\end{align*}
Equating these two identities (in view of \eqref{Spela1}) and  then arguing in a standard way  we see that $d$ cannot be $0$, thus $d=-1$, and moreover we see that $e=0$ and $\rho(A_1)=0$. Returning back to $F(F(x_1,y_0),z_0) = F(x_1,F(y_0,z_0))$ we now also see that $\rho = \omega$ on $A_0$, and that $\lambda(A_1,A_0)=0$ (since $\omega(F(x_1,y_0)) =0$ as $F(x_1,y_0)\in A_1$). 
%Hence also $\lambda(A_0,A_1)=0$ by \eqref{esym}.

Finally we examine $F(F(x_1,y_1),z_1)=F(x_1, F(y_1,z_1))$. By a now familiar method we obtain $\omega(F(x_1,y_1)) = \lambda(x_1,y_1)$.

Let us now summarize what was  proved. 
By  \eqref{dod2} and \eqref{ccc} we have
\begin{equation}\label{z1}
F(y_0,x_0) = - x_0y_0 + \omega(x_0)y_0 + \omega(y_0)x_0 + \omega(F(y_0,x_0)) - \omega(x_0)\omega(y_0)
\end{equation}
for all $x_0,y_0\in A_0$. Next, \eqref{cd1} reduces to 
\begin{equation}\label{z2b}
F(x_1,y_0) = - y_0x_1 + \omega(y_0)x_1 
\end{equation}
for all $y_0\in A_0$, $x_1\in A_1$,  and therefore, by \eqref{Spela2}, 
\begin{equation}\label{z2}
F(y_0,x_1) =  -x_1y_0 +  \omega(y_0)x_1
\end{equation}
for all $y_0\in A_0$, $x_1\in A_1$. Consequently, $ \omega(y_0)x_1\in A_1$, which clearly implies 
that  $\omega(y_0)$ lies in $C_0$ for every $y_0\in A_0$. 
Finally, \eqref{cd3} reduces to
\begin{equation}\label{z3}
F(u_1,z_1) =  z_1u_1 + \omega(F(u_1,z_1)) 
\end{equation}
for all $u_1,z_1\in A_1$.

Let us now define $\tau:A\to C_0$ and $\theta:B\to A_{C_0}$ by 
$$
\tau(b_0) = \omega(\varphi(b_0)), \quad \tau(b_1) =0,\quad \theta(b_0) = \varphi(b_0) - \tau(b_0),\quad \theta(b_1) = \varphi(b_1)
$$
for all $b_0\in B_0$, $b_1\in B_1$. Recalling that $F(x,y) =
\varphi(\varphi^{-1}(x)\varphi^{-1}(y))$ we see that \eqref{z1}-\eqref{z3} imply that $\theta$ is the negative of a superantihomomorphism. The fact that $\tau$ vanishes on supercommutators then immediately follows.
\end{proof}

\section{Third case: grading induced by the exchange automorphism}

In this section we consider the situation where $A$ is the direct product, $A=U\times U$, of two copies of a unital prime associative algebra $U$, and the grading is induced by the  exchange automorphism: $(u,v)^\sigma = (v,u)$. Thus, $A_0= \{(u,u)\,|\, u\in U\}\cong U$ and $A_1= \{(u,-u)\,|\, u\in U\}$. In the case where $U=M_n(F)$, this superalgebra  is denoted by $Q(n)$. 

It is easy to see that the supercenter of $A$ consists of all elements of the form $(z,z)$ where $z$ is in the center of $U$. Thus, the supercenter of $A$ is contained in $A_0$ and is isomorphic to the center  of $A$. Next we note that $(1,-1)\in A_1$ from which it readily follows that $[A_1,A_1]_s = A_0$. All these imply that there are no nonzero graded linear maps $\tau$ on $A$ with the range in the supercenter and vanishing on   $[A,A]_s$. Therefore, saying that a Lie superautomorphism of $A$ is of standard form simply means that it is either a 
 superautomorphism or the negative of a superantiautomorphism. 
 
 \begin{theorem} \label{Texch}
 Let $U$ be a noncommutative prime associative algebra.
 Consider the algebra $A=U\times U$  as a superalgebra with respect to the grading induced by the exchange automorphism. 
 Then every Lie superautomorphism $\varphi$ of $A$ is either a superautomorphism or the negative of a  superantiautomorphism. 
 \end{theorem}

\begin{proof}
Since $\varphi$ is a graded bijective linear map, there exist
 bijective linear maps $\psi,\rho:U\to U$  such that 
 $$\varphi\bigl((u,u)\bigr)= (\psi(u),\psi(u))\quad\mbox{and}\quad\varphi\bigl((u,-u)\bigr)= (\rho(u),-\rho(u)).$$
 Since $(1,-1)$ commutes with every element in $a_0=(u,u)\in A_0$, it readily follows that 
  $\lambda = \rho(1)$ lies in the center of $U$. Of course,  $\lambda$ is nonzero, so it is invertible in the   field of fractions 
  of the center of $U$. Using
  $\varphi(a_1^2) = \varphi(a_1)^2$ with $a_1 = (u,-u)\in A_1$ we see that 
  $ \psi(u^2) = \rho(u)^2.$
   Replacing $u$ by $u+1$  we infer  $\psi(u) = \lambda \rho(u)$. Consequently, 
   $\lambda\rho(u^2) = \rho(u)^2.$
   That is, $u\mapsto \lambda^{-1}\rho(u)$ is a Jordan homomorphism. A well known result by Herstein \cite{Her2} (together with Smiley's extension
   \cite{Sm} covering the 
   characteristic $3$ case)  states that a Jordan homomorphism from a ring onto a prime ring is either a homomorphism or an antihomomorphism. There
   is an apparent technical problem when one wants to apply this theorem to the present setting since the range of our Jordan homomorphism is 
$\lambda^{-1}U$ which may not be a  ring. However, from the proof of this theorem, such as given for example in \cite[pp. 198-199]{FIbook}, it is clear that
the same conclusion holds in this setting. Therefore we have  $\rho(u) = \lambda\theta(u)$, where $\theta$ is either a homomorphism or 
an antihomomorphism. Accordingly, $\psi(u) = \lambda^2\theta(u)$.

Since the restriction of $\varphi$ to $A_0$ is a Lie automorphism, $\psi$ is a Lie automorphism of $U$. From $\psi([u,u']) = [\psi(u),\psi(u')]$ we get
the following: If $\theta$ is a homomorphism, then $(\lambda^4 - \lambda^2)[\theta(u),\theta(u')] =0$, and if  $\theta$ is an  antihomomorphism, 
then $(\lambda^4 + \lambda^2)[\theta(u),\theta(u')] =0$. Since $U$ is assumed to be noncommutative, we have $[\theta(u),\theta(u')]\ne 0$ for some
$u,u'\in U$. Therefore $\lambda^2= 1$ if $\theta$ is a homomorphism, and $\lambda^2= -1$ if $\theta$ is an antihomomorphism.
It can be easily checked that in the first case   
$\varphi$ is a superautomorphism, and in the second case it is the negative of a  superantiautomorphism. 
\end{proof}

The   case where $U$ is commutative must really be excluded.

\begin{example} \label{Exe1}
Let $A=Q(1)$. Pick $\lambda\in F$ with $\lambda\ne 0$ and $\lambda^2 \ne \pm 1$, and define
$\varphi:A\to A$ by 
$$\varphi\bigl((u,v)\bigr) 
=  \Bigl(\frac{\lambda^2 + \lambda}{2}u   + \frac{\lambda^ 2-\lambda}{2}v, \frac{\lambda^2 - \lambda}{2}u   + \frac{\lambda^2+\lambda}{2} v\Bigr).
$$
It is easy to verify that $\varphi$ is a Lie superautomorphism that is not of standard form. 
\end{example}

\section{General central simple associative superalgebras} 
%The next lemma is known even in a more general setting \cite[Lemma 1.5]{M}, but we give a proof for this special case  since it is short and simple.

%\begin{lemma}\label{Lsim}
%Let $A$ be a simple associative superalgebra. If $A$ is not simple as an algebra, then there exists a simple algebra $U$ such that
%$A\cong B\times B$ where  the grading of $B\times B$ is induced by the exchange automorphism. 
%\end{lemma}

%\begin{proof}
%Let $B$ be an ideal of $A$ such that $B\ne 0$ and $B\ne A$. Then $B + B^\sigma$ and $B\cap B^\sigma$ are graded ideals of $A$, and so $B + B^\sigma =A$ and $B\cap B^\sigma =0$. Clearly, this means %that $A$ is isomorphic to the superalgebra $B\times B$ with the exchange automorphism inducing the grading. If $I$ is an ideal of $B$, then $I \oplus I^\sigma$ is a graded ideal of $A$, yielding the %simplicity of $B$.
%\end{proof}

We are now in a position to establish the principal theorem of the paper.

\begin{theorem} \label{Tsimple}
Let $A$ be a central simple associative superalgebra over $F$ such that $\dim_F A \ne 2,4$. Then every Lie superautomorphism $\varphi$ of $A$ is of the form $\varphi= \theta  + \tau$ where $\theta$ is either a superautomorphism   or the negative of a superantiautomorphism of  $A$, and $\tau$ is a linear map from $A$ into $F$  
satisfying $\tau([A,A]_s)  =0$.
\end{theorem}

\begin{proof}
We first remark that in the $\dim_F A= 1$ case the theorem trivially holds as the grading is then trivial and we may take $\theta = \id$ and $\tau = \varphi - \id$. We assume from now on that
$\dim_F A >1$.

Let us first consider the case where $A$ is not simple as an algebra. In this case $A$ is of the form treated in the preceding section - this is well-known, but let us give a short proof for the sake of completness.
 Pick an ideal $U$ of $A$ such that $U\ne 0$ and $U\ne A$. Let 
 $\sigma$ be the automorphism of $A$ inducing the grading on $A$.
 Note that $U + U^\sigma$ and $U\cap U^\sigma$ are graded ideals of $A$, and so $U + U^\sigma =A$ and $U\cap U^\sigma =0$. This readily implies  that $A$ is isomorphic to the superalgebra $U\times U$ with the exchange automorphism inducing the grading. If $I$ is an ideal of the algebra $U$, then $I \oplus I^\sigma$ is a graded ideal of $A$, yielding the simplicity of $U$. Since $A$ is central and $\dim_F A\ne 2$, $U$ is noncommutative. The result therefore follows from Theorem \ref{Texch}. Note that $\tau =0$ in this case. 
 
 We assume from now on that $A$ is simple as an algebra. Until further notice we also assume that $F$ is algebraically closed.
 
Let $Z$ be the center of $A$. 
We claim that $Z_1=0$. Indeed, if $z_1\in Z_1$, then $z_1^2\in Z_0=F$. Since $F$ is algebraically closed it follows that $z_1^2 = \mu^2$ for some $\mu\in F$, yielding $z_1 = \pm \mu\in Z_0$ as the center of a unital simple algebra is a field; hence $z_1=0$.  Thus, $Z=Z_0=F$, i.e.,
 $A$ is central also as an algebra  (not only as a superalgebra). 

Suppose that $\sigma$  is inner. Thus, there is $u\in A$ such that $x^\sigma = uxu^{-1}$, and $A_0 = \{x\in A\,|\, ux=xu\}$, 
$A_1 = \{x\in A\,|\, ux=-xu\}$. As $\sigma^2 = \id$,  we have $u^2\in F$, and therefore
$u^2 =\lambda^2$ for some $\lambda\in F$.  Hence $e=\frac{1}{2}(1-\lambda^{-1}u)$ is an idempotent, and one can easily show that $A_0$ and $A_1$ are given as in \eqref{e}, i.e., the grading on $A$ is induced by an idempotent. 
 Since $A$ is, as a central simple unital algebra, equal to $A_C$ (i.e., the extended centroid $C$ of $A$ is just $F$), and since we have assumed that $\dim_F A\ne 1,4$, we are in a position to use  Theorem \ref{TJerry}, which yields the  desired conclusion.  

Since $A$ is a simple unital algebra, the concept of an inner automorphism on $A$ coincides with the concept of an X-inner automorphism on $A$. We may therefore assume that $\sigma$ is X-outer. The Skolem-Noether Theorem tells us that  $A$ is infinite-dimensional. But then, as a central simple algebra, $A$ cannot be a GPI-algebra (this follows easily
from  the description of prime GPI-rings, cf. \cite[Section 6.1]{BMMb}). Applying Theorem \ref{TXout} we see that $\varphi$ is of the desired form in this case as well.

Now let $F$ be an arbitrary field. By $\overline{F}$ we denote its algebraic closure. Consider the $\overline{F}$-algebra $\overline{A} = A\otimes_F \overline{F}$. Clearly, $\overline{A}$ becomes a superalgebra by defining 
$\overline{A}_i = A_i\otimes_F \overline{F}$, $i=0,1$. It is easy to check that  $\overline{A} $ is both central and simple as a superalgebra, and, of course,  $\dim_{\overline{F}} \overline{A}\ne 2,4$.
Further, $\overline{\varphi} = \varphi\otimes \id$ is a Lie superautomorphism of $\overline{A}$. By what we have proved it follows that $\overline{\varphi}= \overline{\theta}  + \overline{\tau}$ where $\overline{\theta}$ is either a superhomomorphism  or the negative of a superantihomomorphism of $\overline{A}$, and $\overline{\tau}$ is a linear map from $\overline{A}$ into $\overline{F}$  
such that $\overline{\tau}([\overline{A},\overline{A}]_s)  =0$.
Let us only consider the case where $\overline{\theta}$ is the negative of a superantihomomorphism. For every $a\in A$ we  have
$\tau(a) = \overline{\tau}(a\otimes 1) \in \overline{F}$. It only  remains is to show that $\tau(a)$ actually  lies in $F$. Indeed, if this was true, then $\theta= \varphi - \tau$ would be the the negative of a superantihomomorphism of $A$. Suppose, on the contrary, that $\tau(a)\in \overline{F}\setminus{F}$ for some $a\in A$.  Without loss of generality we may assume that $a$ is homogeneous.
Let us consider the case where $a\in A_0$; if $a\in A_1$, then the next argument requires just some rather obvious  modifications. For every $b\in A$ we have 
\begin{align*}
&\varphi(ab)\otimes 1 - 1\otimes \tau(ab)  = \overline{\theta}(ab\otimes 1) = - \overline{\theta}(b\otimes 1)\overline{\theta}(a\otimes 1)\\
 =&  - \bigl(\varphi(b)\otimes 1 - 1\otimes \tau(b)\bigr)\bigl(\varphi(a)\otimes 1 - 1\otimes \tau(a)\bigr),
\end{align*}
and hence
\begin{align*}
&\bigl(\varphi(ab) + \varphi(b)\varphi(a)\bigr)\otimes 1 - \varphi(b)\otimes \tau(a)\\
 =& \varphi(a)\otimes \tau(b) + 1\otimes \bigl( \tau(ab) -\tau(b)\tau(a)\bigr).
\end{align*}
Since $\tau(a)$ and $1$ are linearly independent over $F$, it follows, in particular, that $\varphi(b)$ lies in $F\varphi(a) + F$. Since $\varphi$ is surjective, this means that $\dim_F A\le 2$, contrary to the assumption. 

It remains to show that $\theta$ is bijective. Since it is clearly injective (its kernel is a graded ideal), we only need to prove the surjectivity. If $A$ is not simple as an algebra, then
$\theta = \varphi$ and there is nothing to prove. Assume therefore that $A$ is a simple algebra. Since $\tau(A_1)=0$, we have $\theta(A_1)= \varphi(A_1) =A_1$. Next, 
the range of $\theta$ contains $[\theta(A_0),\theta(A)] = [\varphi(A_0),\varphi(A)] =[A_0,A]$. Accordingly, the range of $\theta$ is a subalgebra of $A$  containing $[A,A]$. But then it is equal to $A$ \cite[Corollary 1]{Her1}.

%The range of $\theta$ is a subalgebra of $A$ containing $[A,A]_s$. But the subalgebra generated by $[A,A]_s$ is equal to $A$. 
\end{proof}

{\bf Acknowledgement}. 
The authors are  grateful to the referee for  careful reading of the
paper and the resulting useful remarks.

\end{document}